\documentclass[11pt,a4paper]{article}

\usepackage[a4paper,top=27mm,bottom=28mm,left=29mm,right=29mm]{geometry}
\usepackage[T1]{fontenc}
\usepackage[utf8]{inputenc}
\usepackage{amsmath,amssymb,amsthm,mathtools}
\usepackage{amsmath,amssymb}
\usepackage{newtxtext,newtxmath}
\usepackage[final]{microtype}
\usepackage{setspace}
\usepackage{indentfirst}
\usepackage{authblk}
\usepackage{booktabs,array}
\usepackage{enumitem}
\allowdisplaybreaks[2]
\numberwithin{equation}{section}
\newtheorem{theorem}{Theorem}
\newtheorem{conjecture}{Conjecture}
\newtheorem{lemma}{Lemma}[section]
\theoremstyle{remark}

\usepackage{cite}
\usepackage[hidelinks]{hyperref}
\hypersetup{
	pdftitle={A Congruence Conjecture of Z.-W. Sun for Reciprocal Central Binomial Sums},
	pdfauthor={Dian-Wang Hu},
	pdfsubject={Number theory and supercongruences}
}

\newcommand{\dd}{\,\mathrm{d}}

\DeclareMathOperator*{\Res}{Res}

\usepackage{titlesec}
\titleformat{\section}[block]
{\centering\normalfont\bfseries\scshape}
{\thesection.}{0.55em}{}
\titlespacing*{\section}{0pt}{2.7ex plus .7ex minus .2ex}{1.4ex plus .2ex}
\titleformat{\subsection}[block]
{\normalfont\bfseries}
{\thesubsection}{0.65em}{}
\titlespacing*{\subsection}{0pt}{2.2ex plus .5ex minus .2ex}{0.9ex plus .2ex}

\makeatletter
\renewenvironment{abstract}{%
	\begin{list}{}{\leftmargin=0pt\rightmargin=0pt}%
		\item\relax\small\noindent\textbf{Abstract.}\ }%
	{\end{list}}
\makeatother

\title{\large\bfseries A Congruence Conjecture of Z.-W. Sun for Reciprocal Central Binomial Sums}
\author{Dian-Wang Hu}
\date{}

\begin{document}
	\maketitle
	\vspace{-2.2em}
	
	\begin{abstract}
		Let $p>5$ be a prime. We prove a conjecture of Z.-W. Sun asserting that
		\[
		\sum_{k=1}^{p-1}\frac{1}{k^4\binom{2k}{k}}
		\equiv \frac{H_{p-1}}{p^3}-\frac{7}{45}pB_{p-5}\pmod{p^2}.
		\]
	    As a consequence, we obtain
		\begin{equation*}
			\sum_{k=1}^{p-1}
			\frac{H_{k-1}^{(2)}}{k^2\binom{2k}{k}}
			\equiv
			\frac{H_{p-1}}{3p^3}
			+\frac{26}{135}pB_{p-5}
			\pmod{p^2},
		\end{equation*}
		an equivalent conjecture of K. Hessami Pilehrood and T. Hessami Pilehrood. 
		Here $H_n^{(m)}$ denotes the generalized harmonic numbers of order $m$,
		with $H_n=H_n^{(1)}$, and $B_n$ denotes the $n$th Bernoulli number.
		The proof combines $p$-adic congruences involving harmonic numbers,
		finite-sum identities, and complex residue calculations via the residue theorem.
	\end{abstract}
	
	
	\medskip
	\section{Introduction}
	For integers $n\ge0$ and $m\ge1$, the generalized harmonic numbers of order $m$ are defined by
	\[
	H_n^{(m)}=\sum_{k=1}^n\frac{1}{k^m},
	\qquad H_n:=H_n^{(1)}.
	\]
	Throughout the paper, an empty sum is understood to be zero.
	The Bernoulli numbers $B_n$ are defined by the generating function
	\[
	\frac{x}{e^x-1}
	=\sum_{n=0}^{\infty}B_n\frac{x^n}{n!}.
	\]
	
	In 1978, Ap\'ery \cite{VanderPoorten} employed the rapidly convergent series
	\begin{equation}\label{eq:apery}
		\sum_{k=1}^{\infty}
		\frac{(-1)^{k-1}}{k^3\binom{2k}{k}}
		=\frac{2}{5}\zeta(3)
	\end{equation}
   to prove the irrationality of $\zeta(3)$. A similar approach also provides
   an alternative proof of the irrationality of $\zeta(2)$ through
   	\begin{equation}\label{eq:zeta2}
		\sum_{k=1}^{\infty}
		\frac{1}{k^2\binom{2k}{k}}
		=\frac{1}{3}\zeta(2).
	\end{equation}
	For any prime $p>5$, Tauraso \cite{Tauraso}  proved the following finite $p$-adic analogues
	of \eqref{eq:apery} and \eqref{eq:zeta2}:
	\[
	\sum_{k=1}^{p-1}
	\frac{(-1)^{k-1}}{k^3\binom{2k}{k}}
	\equiv
	\frac{2H_{p-1}}{5p^2}
	\pmod{p^3}
	\qquad
	and
	\qquad
	\sum_{k=1}^{p-1}
	\frac{1}{k^2\binom{2k}{k}}
	\equiv
	\frac{H_{p-1}}{3p}
	\pmod{p^3}.
	\]
	
	Inspired by 
	\[
	\sum_{k=1}^{\infty}
	\frac{1}{k^4\binom{2k}{k}}
	=\frac{17}{36}\zeta(4)
	\]
	(see \cite[p.~89]{Comtet}), Z.-W. Sun conjectured the following $p$-adic analogue
	in [4, Conjecture 1.1], which is the main result of this paper.
	\begin{theorem}\label{thm:main}
		Let $p>5$ be a prime. Then
		\begin{equation}\label{eq:sun}
			\sum_{k=1}^{p-1}
			\frac{1}{k^4\binom{2k}{k}}
			\equiv
			\frac{H_{p-1}}{p^3}-\frac{7}{45}pB_{p-5}
			\pmod{p^2}.
		\end{equation}
	 \end{theorem}
	 
	K.~Hessami Pilehrood and T.~Hessami Pilehrood \cite[Theorem~2]{pilehrood2012} 
	proved the congruence  \eqref{eq:sun} modulo $p$.
	For any prime $p>5$, they further obtained
	\[
	4\sum_{k=1}^{p-1}
	\frac{1}{k^4\binom{2k}{k}}
	-
	3\sum_{k=1}^{p-1}
	\frac{H_{k-1}^{(2)}}{k^2\binom{2k}{k}}
	\equiv
	\frac{3H_{p-1}}{p^3}
	-\frac{6}{5}pB_{p-5}
	\pmod{p^2}.
	\]
	This led them to propose the following conjecture \cite[Conjecture~1]{pilehrood2012}, which is equivalent to Theorem~\ref{thm:main}, together with the separately verified
	case $p=5$.
  	\begin{conjecture}\label{conj:hp}
		Let $p>3$ be a prime. Then
		\begin{equation}\label{eq:hp}
			\sum_{k=1}^{p-1}
			\frac{H_{k-1}^{(2)}}{k^2\binom{2k}{k}}
			\equiv
			\frac{H_{p-1}}{3p^3}
			+\frac{26}{135}pB_{p-5}
			\pmod{p^2}.
		\end{equation}
	\end{conjecture}
   
   The following congruence for the half-range sum plays a crucial role in the proof of Theorem~\ref{thm:main} and may also be of independent interest.
       \begin{theorem}\label{thm:half}
		Let $p>3$ be a prime. Then
		\begin{equation}\label{eq:half}
			\begin{aligned}
				\sum_{j=1}^{(p-1)/2}
				\frac{\binom{2j}{j}}{j^3}
				\left(
				\frac{1}{j^2}+H_j^{(2)}
				\right)
				&\equiv
				-2\sum_{j=1}^{(p-1)/2}
				\frac{p}
				{j^4\binom{2p-2j}{p-j}}
				\left(
				\frac{1}{j^2}+H_j^{(2)}
				\right)  \\
				&\equiv
				\frac{8}{9}B_{p-5}
				\pmod p.
			\end{aligned}
		\end{equation}
	\end{theorem}
	
    Section~2 collects the congruences involving harmonic numbers and the
	finite-sum identities needed in the proofs of the theorems.
	In Section~3 we prove
	Theorem~\ref{thm:half} using the residue theorem. Theorem~\ref{thm:main}
	is then proved in Section~4.

\medskip
\section{Preliminaries}
\subsection{Congruences involving harmonic numbers}
	Let $p>3$ be a prime. We recall the following congruences from \cite[Theorem~5.1 and Corollary~5.2]{sun2000}:
	\begin{equation}\label{eq:Wolstenholme}
	H_{p - 1} \equiv 0\pmod{p^2},
    \end{equation}
   \begin{equation}\label{eq:H25}
 	H_{p - 1}^{(2)} \equiv
 	H_{p - 1}^{(5)} \equiv 0\pmod{p},
   \end{equation}
  \begin{equation}\label{eq:2.3}
  	H_{p-1}^{(3)} \equiv - \frac{6p^{2}}{5}B_{p - 5}\pmod{p^3},
  \end{equation}
\begin{equation}\label{eq:2.4}
	H_{(p-1)/2}^{(5)} \equiv - 6B_{p - 5}\pmod{p}.
\end{equation}
 For $1\le k\le p-1$, the following $p$-adic expansion holds:
\[
\frac1k\left(1-\frac pk\right)^{-1}
\equiv
\frac1k+\frac p{k^2}+\frac{p^2}{k^3}
+\frac{p^3}{k^4}+\frac{p^4}{k^5}
\pmod{p^5}.
\]
Summing this congruence over $k=1,\ldots,p-1$, we obtain
\begin{equation*}
\begin{aligned}
- H_{p - 1}
&= - \sum_{k = 1}^{p - 1}\frac{1}{p - k}
 = \sum_{k = 1}^{p - 1}\frac{1}{k}\left( 1 - \frac{p}{k} \right)^{- 1}\\
&\equiv H_{p - 1} + pH_{p - 1}^{(2)} + p^{2}H_{p - 1}^{(3)}
 + p^{3}H_{p - 1}^{(4)} + p^{4}H_{p - 1}^{(5)}\pmod{p^5}.
\end{aligned}
\end{equation*}
Combining this with \eqref{eq:H25}, we obtain
\begin{equation}\label{eq:2.5}
\frac{2H_{p - 1}}{p^{2}} + H_{p - 1}^{(3)} \equiv - \frac{H_{p - 1}^{(2)}}{p} - pH_{p - 1}^{(4)}\pmod{p^3}.
\end{equation}

\begin{lemma}\label{lem:harmonic}
Let $p>3$ be a prime. Then the following congruences hold:
\begin{equation}\label{eq:2.6}
\sum_{j = 1}^{(p-1)/2}\frac{(-1)^{j}}{j^{5}} \equiv \frac{1}{2^{5}}H_{(p-1)/2}^{(5)}
\pmod{p},
\end{equation}
\begin{equation}\label{eq:2.7}
\sum_{\substack{1 \leq j \leq (p-1)/2 \\ j \equiv - p\pmod{3}}}\frac{(-1)^{j}}{j^{5}} 
\equiv  \frac{121}{6^{5}}H_{(p-1)/2}^{(5)}
\pmod{p}.
\end{equation}
\end{lemma}

\begin{proof}
\textup{(i)} Note that
\begin{equation*}
\sum_{j = 1}^{p - 1}\frac{(-1)^{j}}{j^{5}} = \sum_{j = 1}^{(p-1)/2}\left( \frac{(-1)^{j}}{j^{5}} + \frac{(-1)^{p - j}}{(p - j)^{5}} \right) \equiv 2\sum_{j = 1}^{(p-1)/2}\frac{(-1)^{j}}{j^{5}}\pmod{p}.
\end{equation*}
On the other hand, using \eqref{eq:H25},
\begin{equation*}
\sum_{j = 1}^{p - 1}\frac{(-1)^{j}}{j^{5}} = 2\sum_{\substack{1 \leq j \leq p - 1 \\ 2\mid j}}\frac{1}{j^{5}}-\sum_{j = 1}^{p - 1}\frac{1}{j^{5}} = \frac{2}{2^{5}}\sum_{k = 1}^{(p-1)/2}\frac{1}{k^{5}} - H_{p - 1}^{(5)} \equiv \frac{2}{2^{5}}H_{(p-1)/2}^{(5)}\pmod{p}.
\end{equation*}
Comparing these two expressions gives \eqref{eq:2.6}.

\medskip
\textup{(ii)} For each $1\le j\le(p-1)/2$, let $r_j$ and $\varepsilon_j$ be uniquely
determined by
\[
6j\equiv \varepsilon_j r_j\pmod p,
\qquad
1\le r_j\le\frac{p-1}{2},\qquad \varepsilon_j\in\{\pm1\}.
\]
The map $j\mapsto r_j$ is a permutation of
$\{1,\ldots,(p-1)/2\}$. Indeed, 
$r_{j_1}=r_{j_2}$ implies
$j_1\equiv\pm j_2\pmod p$; the minus sign is impossible since
$2\le j_1+j_2\le p-1$. Hence $j_1=j_2$, so the map is injective
and therefore bijective.

For $1\le r\le(p-1)/2$, define $\mu_r:=\varepsilon_j$, where $j$
is the unique index such that $r_j=r$. Thus, reindexing yields
\begin{equation}\label{eq:2.8}
\frac{1}{6^5}H_{(p-1)/2}^{(5)}= \sum_{j = 1}^{(p-1)/2}\frac{1}{(6j)^{5}}
\equiv
\sum_{r=1}^{(p-1)/2}\frac{\mu_r}{r^5}
\pmod p.
\end{equation}

For the unique $j$ satisfying $r_j=r$, we have
$6j\equiv\mu_r r\pmod p$ and $6\le6j\le3p-3$; hence
\[
6j\in\{r,p-r,p+r,2p-r,2p+r,3p-r\}.
\]
Since $p\equiv\pm1\pmod6$, exactly one element of the above set is divisible
by $6$, which determines $\mu_r$.
Examining the six residue classes modulo $6$, we obtain
\[
\mu_r=
\begin{cases}
-(-1)^r, & r\equiv-p\pmod3,\\
(-1)^r, & r\not\equiv-p\pmod 3.
\end{cases}
\]
Subtracting \eqref{eq:2.8} from \eqref{eq:2.6}, we obtain
\begin{equation*}
	\left( \frac{1}{2^{5}} - \frac{1}{6^{5}} \right)H_{(p-1)/2}^{(5)} \equiv \sum_{r = 1}^{(p-1)/2}\frac{(-1)^{r} - \mu_{r}}{r^{5}} = 2\sum_{\substack{1 \leq r \leq (p-1)/2 \\ r \equiv - p\pmod{3}}}\frac{(-1)^{r}}{r^{5}}\pmod{p}.
\end{equation*}
This proves \eqref{eq:2.7}.
\end{proof}

\medskip
\subsection{Finite-sum identities}

Define
\begin{equation}\label{eq:2.9}
P_0(x):=1,\qquad
P_{k-1}(x):=\prod_{j=1}^{k-1}\left(1-\frac{x^2}{j^2}\right),
\qquad k\geq2.
\end{equation}

\begin{lemma}\label{lem:finite-sum}
Let $n\ge1$ be an integer and
$x\in\mathbb{R}\setminus\{\pm1,\ldots,\pm n\}$. Then
\begin{equation}\label{eq:2.10}
\begin{aligned}
\frac{1}{2}\sum_{k = 1}^{n}\frac{3k^{2} + x^{2}}{k^{2}(k^{2} - x^{2})\binom{2k}{k}}P_{k - 1}(x)
+ \frac{1}{2}\sum_{k = 1}^{n}\frac{(-1)^{n + k}}{k^{2}\binom{n}{k}\binom{n + k}{k}}P_{k - 1}(x)
= \sum_{k = 1}^{n}\frac{(-1)^{k}}{x^{2} - k^{2}}.
\end{aligned}
\end{equation}
\end{lemma}

\begin{proof}
Let $L_n(x)$ and $R_n(x)$ denote the left- and right-hand sides of \eqref{eq:2.10}.
For $n=1$,
\[
L_1(x)=R_1(x)=\frac{1}{1-x^2},
\]
For $n\ge2$,
\[
R_n(x)-R_{n-1}(x)
=\frac{(-1)^n}{x^2-n^2}.
\]
Thus it suffices to prove
\begin{equation}\label{eq:2.11}
L_{n}(x) - L_{n - 1}(x) = \frac{(-1)^{n}}{x^{2} - n^{2}} \qquad n\ge2.
\end{equation}

We prove \eqref{eq:2.11} by a telescoping argument. For
\(1\le k\le n\), set
\begin{equation}\label{eq:2.12}
\begin{alignedat}{2}
B_{n,k}(x)&:=\frac{(-1)^{n+k}}{k^{2}\binom{n}{k}\binom{n+k}{k}}P_{k-1}(x),
&\qquad
G_{n,k}(x)&:=\frac{2n(n+k)}{x^{2}-n^{2}}B_{n,k}(x).
\end{alignedat}
\end{equation}
For $1\le k\le n-1$, we have
\[
\frac{B_{n,k}(x)}{B_{n-1,k}(x)}
=-\frac{n-k}{n+k},
\qquad
\frac{B_{n,k+1}(x)}{B_{n,k}(x)}
=
\frac{x^2-k^2}{(n-k)(n+k+1)}.
\]
Hence
\begin{equation*}\label{eq:2.13}
B_{n,k}(x)-B_{n-1,k}(x)
=\frac{2n}{n-k}B_{n,k}(x)
=G_{n,k+1}(x)-G_{n,k}(x).
\end{equation*}
Summing this identity over $k=1,\ldots,n-1$, we obtain
\begin{equation*}
\sum_{k = 1}^{n - 1}\left( B_{n,k}(x) - B_{n - 1,k}(x) \right) = G_{n,n}(x) - G_{n,1}(x).
\end{equation*}
Consequently,
\begin{equation}\label{eq:2.14}
\sum_{k=1}^{n}B_{n,k}(x)-\sum_{k=1}^{n-1}B_{n-1,k}(x)
=B_{n,n}(x)+G_{n,n}(x)-G_{n,1}(x).
\end{equation}
By the definitions in \eqref{eq:2.12},
\begin{equation}\label{eq:2.15}
B_{n,n}(x)=\frac{P_{n-1}(x)}{n^{2}\binom{2n}{n}},\qquad
G_{n,n}(x)=-\frac{4P_{n-1}(x)}{(n^{2}-x^{2})\binom{2n}{n}},\qquad
G_{n,1}(x)= \frac{2(-1)^{n + 1}}{x^{2} - n^{2}}.
\end{equation}
Using \eqref{eq:2.14} and \eqref{eq:2.15}, we obtain
\begin{equation*}
\begin{aligned}
L_{n}(x) - L_{n - 1}(x)
&= \frac{1}{2}\frac{3n^{2} + x^{2}}{n^{2}\left( n^{2} - x^{2} \right)\binom{2n}{n}}P_{n - 1}(x)
 + \frac{1}{2}\left( \sum_{k = 1}^{n}{B_{n,k}(x)} - \sum_{k = 1}^{n - 1}{B_{n - 1,k}(x)} \right)\\
&= \frac{1}{2}\frac{3n^{2} + x^{2}}{n^{2}\left( n^{2} - x^{2} \right)\binom{2n}{n}}P_{n - 1}(x)
 + \frac{1}{2}\left( G_{n,n}(x) - G_{n,1}(x) + B_{n,n}(x) \right)\\
&= \frac{1}{2}\frac{P_{n - 1}(x)}{\binom{2n}{n}}\left[ \frac{3n^{2} + x^{2}}{n^{2}\left( n^{2} - x^{2} \right)} + \frac{1}{n^{2}} - \frac{4}{n^{2} - x^{2}} \right]
 - \frac{1}{2}G_{n,1}(x)\\
&= \frac{(-1)^{n}}{x^{2} - n^{2}}.
\end{aligned}
\end{equation*}
Thus $L_n(x)$ and $R_n(x)$ have the same initial value and satisfy
the same recurrence, so $L_n(x)=R_n(x)$ for all $n\ge1$.
This proves \eqref{eq:2.10}.
\end{proof}

\begin{lemma}\label{lem:En}
Let \(n>1\) be an integer with \(n\equiv\pm1\pmod{6}\). Then
\begin{equation}\label{eq:2.17}
E_{n}:= \sum_{k = 1}^{n - 1}\frac{1}{n^{2} - k^{2}} = n\sum_{k = 1}^{n - 1}(-1)^{k - 1}\frac{\binom{n + k - 1}{2k - 1}}{k\left( k^{2} - n^{2} \right)}.
\end{equation}
\end{lemma}

\begin{proof}
Applying Lemma~\ref{lem:finite-sum} with $n$ replaced by $n-1$ and $x=n$, we set
\begin{equation*}
A:= \frac{1}{2}\sum_{k = 1}^{n - 1}\frac{3k^{2} + n^{2}}{k^{2}\left( k^{2} - n^{2} \right)}\frac{P_{k - 1}(n)}{\binom{2k}{k}},\ \ 
B:= \frac{1}{2}\sum_{k = 1}^{n - 1}\frac{(-1)^{n - 1 + k}}{k^{2}\binom{n - 1}{k}\binom{n - 1 + k}{k}}P_{k - 1}(n),\ \ 
C:= \sum_{k = 1}^{n - 1}\frac{(-1)^{k}}{n^{2} - k^{2}}.
\end{equation*}
By \eqref{eq:2.10}, we have
\begin{equation}\label{eq:2.19}
A + B = C.
\end{equation}
Note that
\begin{equation*}
P_{k - 1}(n) = \prod_{j = 1}^{k - 1}\left( 1 - \frac{n^{2}}{j^{2}} \right) = \prod_{j = 1}^{k - 1}\frac{j^{2} - n^{2}}{j^{2}} = (-1)^{k - 1}\frac{\prod_{j = 1}^{k - 1}\left( n^{2} - j^{2} \right)}{\left( (k - 1)! \right)^{2}},
\end{equation*}
\begin{equation*}
k^{2}\binom{n - 1}{k}\binom{n - 1 + k}{k} = \frac{k^{2}}{(k!)^{2}}\left( \prod_{j = 1}^{k}(n - j) \right)\left( \prod_{j = 0}^{k - 1}(n + j) \right) = n(n - k)\frac{\prod_{j = 1}^{k - 1}\left( n^{2} - j^{2} \right)}{\left( (k - 1)! \right)^{2}}.
\end{equation*}
Taking the ratio of these two identities gives
\begin{equation}\label{eq:2.20}
\frac{P_{k - 1}(n)}{k^{2}\binom{n - 1}{k}\binom{n - 1 + k}{k}} = \frac{(-1)^{k - 1}}{n(n - k)}.
\end{equation}
Since $n\equiv\pm1\pmod6$, $n$ is odd, and \eqref{eq:2.20} gives
\begin{equation*}
B = \frac{1}{2}\sum_{k = 1}^{n - 1}\frac{(-1)^{n - 1 + k}}{k^{2}\binom{n - 1}{k}\binom{n - 1 + k}{k}}P_{k - 1}(n) = - \frac{1}{2n}\sum_{k = 1}^{n - 1}\frac{1}{n - k} = - \frac{H_{n - 1}}{2n},
\end{equation*}
On the other hand,
\begin{equation*}
E_{n} - C
= \sum_{k = 1}^{n - 1}\frac{1 - (-1)^{k}}{n^{2} - k^{2}}
= \frac{1}{n}\sum_{k \in \{ 1,3,\ldots,n - 2\}}\left( \frac{1}{n - k} + \frac{1}{n + k} \right)\\
= \frac{1}{2n}\sum_{j = 1}^{n - 1}\frac{1}{j}
 = - B.
\end{equation*}
Combining this with \eqref{eq:2.19}, we obtain
\begin{equation}\label{eq:2.21}
E_{n} = C - B = A.
\end{equation}

Define
\begin{equation*}\label{eq:2.22}
a_{k}:= (-1)^{k - 1}\binom{n - 1 + k}{2k - 1}.
\end{equation*}
Using the combinatorial identity $\binom{n}{m}\binom{m}{k} 
=\binom{n}{k}\binom{n-k}{m-k}$ and \eqref{eq:2.20},
we obtain
\begin{equation*}
\frac{P_{k - 1}(n)}{k^{2}\binom{2k}{k}\binom{n - 1 + k}{2k}} = \frac{P_{k - 1}(n)}{k^{2}\binom{n - 1}{k}\binom{n - 1 + k}{k}} = \frac{(-1)^{k - 1}}{n(n - k)}.
\end{equation*}
Hence
\begin{equation*}\label{eq:2.23}
\frac{P_{k - 1}(n)}{k^{2}\binom{2k}{k}} = \frac{(-1)^{k - 1}}{n(n - k)}\binom{n - 1 + k}{2k} = \frac{(-1)^{k - 1}}{2nk}\binom{n - 1 + k}{2k - 1} = \frac{a_{k}}{2nk}.
\end{equation*}
Let \(R\) denote the right-hand side of \eqref{eq:2.17}. 
\begin{equation*}
\begin{aligned}
A - R
= \sum_{k = 1}^{n - 1}\frac{a_{k}}{k\left( k^{2} - n^{2} \right)}\left( \frac{3k^{2} + n^{2}}{4n} - n \right)
= \frac{3}{4n}\sum_{k = 1}^{n - 1}\frac{a_{k}}{k}.
\end{aligned}
\end{equation*}
The identity in \cite[Section.~4.2]{prudnikov1986} gives
\begin{equation*}
\sum_{k = 1}^{n}\frac{a_{k}}{k} = \sum_{k = 1}^{n}\frac{(-1)^{k - 1}}{k}\binom{n - 1 + k}{2k - 1} = \frac{2}{n}\left( 1 - \cos\frac{n\pi}{3} \right).
\end{equation*}
Since \(n\equiv\pm1\pmod{6}\), we have
$a_n=1$ and $\cos(n\pi/3)=1/2$. Hence
\begin{equation*}
A - R = \frac{3}{4n}\left( \sum_{k = 1}^{n}\frac{a_{k}}{k} - \frac{a_{n}}{n} \right) = \frac{3}{4n}\left[ \frac{2}{n}\left( 1 - \cos\frac{n\pi}{3} \right) - \frac{1}{n} \right] = 0.
\end{equation*}
Together with \eqref{eq:2.21}, this gives
\begin{equation*}
E_{n} = A = R.
\end{equation*}
This proves the lemma.
\end{proof}

\begin{lemma}\label{lem:Fn}
For every integer \(n\ge 1\), define
\begin{equation}\label{eq:2.24}
F_{n}(x) := \sum_{k = 0}^{\left\lfloor (n-1)/2 \right\rfloor}(-1)^{k}\binom{n - 1 - k}{k}x^{k},
\end{equation}
where $\lfloor\cdot\rfloor$ denotes the floor function. 
If $t\neq1/2$, then
\begin{equation}\label{eq:2.25}
F_{n}\left( t(1 - t) \right) = \frac{(1 - t)^{n} - t^{n}}{1 - 2t}.
\end{equation}
\end{lemma}

\begin{proof}
Pascal's identity yields the recurrence
\[
F_n(x)=F_{n-1}(x)-xF_{n-2}(x),
\qquad n\ge3,
\]
with initial values $F_1(x)=F_2(x)=1$. After setting $x=t(1-t)$, the characteristic
polynomial factors as
\[
\lambda^2-\lambda+t(1-t)
=(\lambda-t)(\lambda-(1-t)).
\]
Since $t\ne1/2$, the initial conditions give
\[
F_n(t(1-t))=\frac{(1-t)^n-t^n}{1-2t}.
\]
\end{proof}

We recall the Chebyshev polynomials of the second kind, defined by
\begin{equation}\label{eq:2.29}
	U_n(\cos\theta)
	=
	\frac{\sin((n+1)\theta)}{\sin\theta},
    \qquad n\ge0.
\end{equation}
They admit the following explicit expansion:
\begin{equation}\label{eq:2.28}
	U_n(x)
	=
	\sum_{k=0}^{\lfloor n/2\rfloor}
	(-1)^k
	\binom{n-k}{k}
	(2x)^{\,n-2k}.
\end{equation}

Let
\begin{equation*}
\xi := e^{\pi i/3},\qquad \overline{\xi} := e^{- \pi i/3},\qquad \omega := e^{2\pi i/3}.
\end{equation*}

\begin{lemma}\label{lem:roots}
For a nonnegative integer $m$, we have
\begin{equation}\label{eq:2.30}
A_{m} := \sum_{k = 1}^{m}\frac{(-1)^{k}}{k}\binom{2m - k}{k} = \sum_{k = 1}^{2m}\frac{\xi^k+\overline{\xi}^{\,k}}{k} - H_{2m},
\end{equation}
\begin{equation}\label{eq:2.31}
B_{m} := \sum_{k = 0}^{m}\frac{(-1)^{k}}{2m + 1 - k}\binom{2m - k}{k} = \sum_{k = 1}^{2m + 1}\frac{\xi^k+\overline{\xi}^{\,k}}{k},
\end{equation}
\begin{equation}\label{eq:2.32}
C_{m}:= \sum_{k = 0}^{m}{\frac{(-1)^{k}}{2m + 1 - 2k}\binom{2m - k}{k}} 
=\frac{2\cos\left( \frac{(2m + 1)\pi}{3} \right)}{2m + 1}
= \frac{\xi^{2m+1}+\overline{\xi}^{\,2m+1}}{2m + 1}.
\end{equation}
Combining \eqref{eq:2.30}--\eqref{eq:2.32}, we obtain
\begin{equation}\label{eq:2.33}
A_{m} - B_{m} + C_{m} = - H_{2m}.
\end{equation}
\end{lemma}

\begin{proof}
The case $m=0$ is immediate. Hence assume $m\ge1$ and set
\begin{equation*}
c_{m,k}:= (-1)^{k}\binom{2m - k}{k}, \qquad 0\le k\le m.
\end{equation*}

\textup{(i)} Applying \eqref{eq:2.24} with $n=2m+1$, we obtain
\begin{equation*}
\frac{F_{2m + 1}(x) - 1}{x} = \sum_{k = 1}^{m}(-1)^{k}\binom{2m - k}{k}x^{k - 1} = \sum_{k = 1}^{m}{c_{m,k}x^{k - 1}}.
\end{equation*}
Hence
\begin{equation}\label{eq:2.34}
A_{m} = \sum_{k = 1}^{m}{\frac{c_{m,k}}{k}} = \int_{0}^{1}{\sum_{k = 1}^{m}c_{m,k}x^{k - 1}\dd x} = \int_{0}^{1}{\frac{F_{2m + 1}(x) - 1}{x}\dd x}.
\end{equation}
Since $(F_{2m+1}(x)-1)/x$ is a polynomial, the integrand is
entire. Hence the integral depends only on its endpoints.
To apply \eqref{eq:2.25}, we use the parametrization
\[
x=t(1-t),\qquad t:0\to\xi=e^{\pi i/3},
\]
where $t$ traverses the line segment from $0$ to $\xi$. Therefore,
\begin{equation*}
x(0)=0,\qquad x(\xi)=\xi(1-\xi)=1.
\end{equation*}
Substituting $x = t(1 - t)$ into \eqref{eq:2.34} and using \eqref{eq:2.25}, we obtain
\begin{equation*}
\begin{aligned}
A_m&=\int_0^{\xi}\frac{(1-t)^{2m+1}-t^{2m+1}-(1-2t)}{t(1-t)}\,\dd t 
=\int_0^{\xi}\left(\frac{1-t^{2m}}{1-t}-\frac{1-(1-t)^{2m}}{t}\right)\dd t \\
&=\int_0^{\xi}\left(\sum_{k=0}^{2m-1}t^k-\sum_{k=0}^{2m-1}(1-t)^k\right)\dd t
=\sum_{k=1}^{2m}\frac{\xi^k+\overline{\xi}^{\,k}}{k}-H_{2m}.
\end{aligned}
\end{equation*}
This proves \eqref{eq:2.30}.

\medskip
\textup{(ii)} For $B_m$, termwise integration gives
\begin{equation}\label{eq:2.35}
B_{m} = \sum_{k = 0}^{m}\frac{c_{m,k}}{2m + 1 - k} = \int_{0}^{1}{\sum_{k = 0}^{m}c_{m,k}x^{2m - k}\,\dd x} = \int_{0}^{1}{x^{2m}F_{2m + 1}\left( \frac{1}{x} \right)\,\dd x}.
\end{equation}
Since the integrand is again a polynomial, the path may be deformed, and we use
\begin{equation*}
	x = \frac{(1 + u)^{2}}{u},\quad\quad u:\  - 1 \rightarrow \omega = e^{2\pi i/3}.
\end{equation*}
where  \(u\) traverses the shorter arc of the unit circle from \(-1\) to \(\omega\). Then 
\begin{equation*}
	x(-1)=0,\qquad x(\omega)=(1+\omega)^2/\omega=1.
\end{equation*}
For $u\ne-1$ on this arc, put $t=1/(1+u)$ in \eqref{eq:2.25}. 
Then
 \begin{equation}\label{eq:2.36}
 	\begin{aligned}
 		x^{2m}F_{2m+1}\left(\frac1x\right)
 		=\frac{(1+u)^{4m}}{u^{2m}}F_{2m+1}\left(\frac{1}{1+u}\frac{u}{1+u}\right)
 	 		=\frac{(1+u)^{2m}}{u^{2m}(u-1)}\left(u^{2m+1}-1\right).
 	\end{aligned}
 \end{equation}
 Since $m\ge1$, the right-hand side of \eqref{eq:2.36} extends
 holomorphically to $u=-1$. Substituting $x = (1 + u)^{2}/u$ into \eqref{eq:2.35} 
 and using \eqref{eq:2.36}, we obtain
\begin{equation*}
B_{m} = \int_{0}^{1}{x^{2m}F_{2m + 1}\left( \frac{1}{x} \right)\dd x} = \int_{- 1}^{\omega}\frac{\left( u^{2m + 1} - 1 \right)(1 + u)^{2m + 1}}{u^{2m + 2}}\, \dd u.
\end{equation*}
Next, observe that
\begin{equation*}
\begin{aligned}
&\frac{\dd}{\dd u}\left(\sum_{k=1}^{2m+1}\binom{2m+1}{k}\frac{u^k+u^{-k}}{k}\right)
=\sum_{k=1}^{2m+1}\binom{2m+1}{k}\left(u^{k-1}-u^{-k-1}\right)\\
=&\frac1u\left((1+u)^{2m+1}-\left(1+\frac1u\right)^{2m+1}\right)
=\frac{(u^{2m+1}-1)(1+u)^{2m+1}}{u^{2m+2}}.
\end{aligned}
\end{equation*}
Hence
\begin{equation*}
B_{m} = \left[ \sum_{k = 1}^{2m + 1}\binom{2m + 1}{k}\frac{u^{k} + u^{- k}}{k} \right]_{- 1}^{\omega}
= \sum_{k = 1}^{2m + 1}\binom{2m + 1}{k}\frac{\omega^{k} + \omega^{- k} - 2(-1)^{k}}{k}.
\end{equation*}
We use the elementary identity
\begin{equation*}
\sum_{k=1}^{n}\binom{n}{k}\frac{z^k}{k}
=
\sum_{k=1}^{n}\frac{(1+z)^k-1}{k}.
\end{equation*}
which follows by differentiating both sides and comparing the values at $z=0$.
Applying the identity with $n=2m+1$ and
$z=\omega,\omega^{-1},-1$, respectively, gives
\begin{equation*}
\begin{aligned}
B_m
&=\sum_{k=1}^{2m+1}\binom{2m+1}{k}\frac{\omega^k}{k}
 +\sum_{k=1}^{2m+1}\binom{2m+1}{k}\frac{\omega^{-k}}{k}
 -2\sum_{k=1}^{2m+1}\binom{2m+1}{k}\frac{(-1)^k}{k}\\
&=\sum_{k=1}^{2m+1}\frac{(1+\omega)^k-1}{k}
 +\sum_{k=1}^{2m+1}\frac{(1+\omega^{-1})^k-1}{k}
 -2\sum_{k=1}^{2m+1}\frac{(1-1)^k-1}{k}\\
&=\sum_{k=1}^{2m+1}\frac{\xi^k+\overline{\xi}^{\,k}}{k}.
\end{aligned}
\end{equation*}
Here we used $1+\omega=\xi$ and
$1+\omega^{-1}=\overline{\xi}$. This proves \eqref{eq:2.31}.

\medskip
\textup{(iii)} Using \eqref{eq:2.28}, we obtain
\begin{equation*}
C_{m} = \sum_{k = 0}^{m}\frac{c_{m,k}}{2m + 1 - 2k} = \int_{0}^{1}{\sum_{k = 0}^{m}{c_{m,k}x^{2m - 2k}}}\, \dd x = \int_{0}^{1}{U_{2m}\left( \frac{x}{2} \right)}\, \dd x.
\end{equation*}
With the substitution \(x=2\cos\theta\), \eqref{eq:2.29} gives
\begin{equation*}
\begin{aligned}
C_m
&=-\int_{\pi/2}^{\pi/3}U_{2m}(\cos\theta)(2\sin\theta)\dd\theta
=2\int_{\pi/3}^{\pi/2}\sin((2m+1)\theta)\dd\theta\\
&=\frac{2}{2m+1}\cos\frac{(2m+1)\pi}{3}
=\frac{\xi^{2m+1}+\overline{\xi}^{\,2m+1}}{2m+1}.
\end{aligned}
\end{equation*}
This proves \eqref{eq:2.32}. Equation \eqref{eq:2.33} now follows
from \eqref{eq:2.30}--\eqref{eq:2.32}, completing the proof.
\end{proof}

\medskip
\section{Proof of Theorem~\ref{thm:half}}

Let $p>3$ be a prime and set
\[
S:=
\sum_{j=1}^{(p-1)/2}
\frac{\binom{2j}{j}}{j^3}
\left(\frac{1}{j^2}+H_j^{(2)}\right).
\]
For $1\le j\le(p-1)/2$, define
\[
\varphi_j(z):=
\frac{1}{j^2-z}
\prod_{k=1}^{j}\frac{1}{k^2-z},
\]
and set
\begin{equation}\label{eq:3.1}
F(z):=2\sum_{j=1}^{(p-1)/2}(2j-1)!\,\varphi_j(z).
\end{equation}
Expanding the factors in \eqref{eq:3.1} around $z=0$, the coefficient of $z$
in the Taylor expansion of $F(z)$ is exactly $S$. Hence the residue of $F(z)/z^2$ at $z=0$ is
\[
\Res_{z=0}\frac{F(z)}{z^2}=S.
\]
The only other finite poles of $F(z)/z^2$ are $z=t^2$,
$1\le t\le(p-1)/2$. Moreover,
\[
\frac{F(z)}{z^2}=O(z^{-4}),
\qquad|z|\to\infty,
\]
and hence the residue at infinity vanishes. The residue theorem therefore yields
\begin{equation}\label{eq:3.5}
S=-\sum_{t=1}^{(p-1)/2}\Res_{z=t^2}\frac{F(z)}{z^2}.
\end{equation}

\subsection{The residue at \texorpdfstring{\(z=t^2\)}{z=t\string^2}}

Fix $1\le t\le (p-1)/2$. Since $\varphi_j(z)/z^2$ is holomorphic at
$z=t^2$ for $j<t$, only the summands with $j\ge t$ contribute to the
residue of $F(z)/z^2$ at $z=t^2$: the term with $j=t$ has a double pole, whereas those
with $j>t$ have simple poles.

\textup{(i)} If \(j=t\), then
\begin{equation}\label{eq:3.7}
	\varphi_t(z)
	=
	\frac{1}{t^2-z}\prod_{k=1}^{t}\frac{1}{k^2-z}
	=
	\frac{1}{(z-t^2)^2}
	\prod_{k=1}^{t-1}\frac{1}{k^2-z}.
\end{equation}
Set
\begin{equation}\label{eq:3.8}
	\alpha(z)
	:=
	\frac{1}{z^2}
	\prod_{k=1}^{t-1}\frac{1}{k^2-z}.
\end{equation}
Since $\alpha$ is holomorphic near $t^2$, the residue of the double
pole is $\alpha'(t^2)$.
From \eqref{eq:3.8},
\begin{equation*}
	\alpha\left( t^{2} \right) = \frac{1}{t^{4}}\prod_{k = 1}^{t - 1}\frac{1}{k^{2} - t^{2}} = \frac{(-1)^{t - 1}}{t^{4}}\prod_{k = 1}^{t - 1}\frac{1}{(t - k)(t + k)} = \frac{(-1)^{t-1}}{t^3(2t-1)!}.
\end{equation*}
Logarithmically differentiating \eqref{eq:3.8}, we obtain
\begin{equation*}
	\alpha'(z) = \alpha(z)\left( - \frac{2}{z} + \sum_{k = 1}^{t - 1}\frac{1}{k^{2} - z} \right) = - \alpha(z)\left( \frac{2}{z} + \sum_{k = 1}^{t - 1}\frac{1}{z - k^{2}} \right).
\end{equation*}
At \(z=t^2\),
\begin{equation*}
	\frac{2}{t^{2}} + \sum_{k = 1}^{t - 1}\frac{1}{t^{2} - k^{2}} = \frac{2}{t^{2}} + \frac{1}{2t}\sum_{k = 1}^{t - 1}\left( \frac{1}{t - k} + \frac{1}{t + k} \right) = \frac{2}{t^{2}} + \frac{1}{2t}\left( H_{2t} - \frac{1}{t} - \frac{1}{2t} \right) = \frac{5}{4t^{2}} + \frac{H_{2t}}{2t}.
\end{equation*}
Hence
\begin{equation}\label{eq:3.10}
\begin{aligned}
\alpha'(t^2)
=-\alpha(t^2)\left(\frac{2}{t^2}+\sum_{k=1}^{t-1}\frac{1}{t^2-k^2}\right)
=\frac{1}{(2t-1)!}
\left(\frac{5(-1)^t}{4t^5}+\frac{(-1)^tH_{2t}}{2t^4}\right).
\end{aligned}
\end{equation}
Using \eqref{eq:3.7}--\eqref{eq:3.10}, 
the residue contributed by the term with index $t$ is
\begin{equation}\label{eq:3.11}
\begin{aligned}
\Res_{z=t^2}\frac{2(2t-1)!\varphi_t(z)}{z^2}
=2(2t-1)!\alpha'(t^2)
=\frac{5(-1)^t}{2t^5}+\frac{(-1)^tH_{2t}}{t^4}.
\end{aligned}
\end{equation}

\textup{(ii)} If \(j>t\), then
\begin{equation}\label{eq:3.12}
	\varphi_j(z)
	=
	\frac{1}{z-t^2}
	\frac{-1}{j^2-z}
	\prod_{\substack{1\le k\le j\\ k\ne t}}
	\frac{1}{k^2-z}.
\end{equation}
Set
\begin{equation}\label{eq:3.13}
	\beta(z)
	:=
	\frac{1}{z^2}
	\frac{-1}{j^2-z}
	\prod_{\substack{1\le k\le j\\ k\ne t}}
	\frac{1}{k^2-z}.
\end{equation}
Substituting $z=t^2$, we obtain
\begin{equation}\label{eq:3.14}
	\beta\left( t^{2} \right) = \frac{1}{t^{4}}\frac{- 1}{j^{2} - t^{2}}\prod_{\substack{1\le k\le j\\k\ne t}}\frac{1}{k^{2} - t^{2}} = \frac{(-1)^{t}2}{t^{2}\left( j^{2} - t^{2} \right)}\frac{1}{(j - t)!(j + t)!}.
\end{equation}
Since $\beta$ is holomorphic at $t^2$, \eqref{eq:3.12}--\eqref{eq:3.14} give 
the residue contributed by the term with index $j$ 
\begin{equation}\label{eq:3.15}
\Res_{z=t^2}\frac{2(2j-1)!\,\varphi_j(z)}{z^2}
=2(2j-1)!\beta\left( t^{2} \right) 
=\frac{2(-1)^t}{jt^2(j^2-t^2)}\binom{2j}{j-t}.
\end{equation}

Combining the contributions in \eqref{eq:3.11} and \eqref{eq:3.15}, we obtain
\begin{equation}\label{eq:3.16}
\begin{aligned}
\Res_{z=t^2}\frac{F(z)}{z^2}
&=\frac{5(-1)^t}{2t^5}
+\frac{2(-1)^t}{t^2}
\left(
\frac{H_{2t}}{2t^2}
+\sum_{j=t+1}^{(p-1)/2}
\frac{1}{j(j^2-t^2)}\binom{2j}{j-t}
\right).
\end{aligned}
\end{equation}
Set \(m=(p-1)/2-t\). Then
\begin{equation}\label{eq:t}
	t = \frac{p - 1}{2} - m \equiv - \frac{2m + 1}{2}\pmod{p}
\end{equation}
and
\begin{equation}\label{eq:H_{2t}}
	H_{2t} = H_{p - 1 - 2m} = H_{p - 1} - \sum_{k = 1}^{2m}\frac{1}{p - k} \equiv H_{2m}\pmod{p}.
\end{equation}
We first note that
\[
\binom{2t+2k}{k}
=
\binom{p-(2m+1-2k)}{k}
\equiv
(-1)^k\binom{2m-k}{k}
\pmod p.
\]
Therefore, after reindexing by $j=t+k$ and using \eqref{eq:t}, we obtain
\[
\begin{aligned}
	\sum_{j=t+1}^{(p-1)/2}
	\frac{1}{j(j^2-t^2)}
	\binom{2j}{j-t}
	&=
	\sum_{k=1}^{m}
	\frac{1}{(t+k)k(2t+k)}
	\binom{2t+2k}{k}\\
	&\equiv
	2\sum_{k=1}^{m}
	\frac{(-1)^k}
	{k(2m+1-k)(2m+1-2k)}
	\binom{2m-k}{k}
	\pmod p.
\end{aligned}
\]
Using the partial-fraction decomposition
\[
\frac{1}
{k(2m+1-k)(2m+1-2k)}
=
\frac{1}{(2m+1)^2}
\left(
\frac{1}{k}
-
\frac{1}{2m+1-k}
+
\frac{4}{2m+1-2k}
\right),
\]
This decomposition expresses the sum in terms of the quantities
$A_m$, $B_m$, and $C_m$ introduced in Lemma~\ref{lem:roots}. 
Hence, by Lemma~\ref{lem:roots} and \eqref{eq:t}--\eqref{eq:H_{2t}}, we obtain
\begin{equation}\label{eq:key}
	\begin{aligned}
		\sum_{j=t+1}^{(p-1)/2}
		\frac{1}{j(j^2-t^2)}
		\binom{2j}{j-t}
		&\quad\equiv
		\frac{2}{(2m+1)^2}
		\left[
		A_m
		-
		\left(B_m-\frac{1}{2m+1}\right)
		+
		4\left(C_m-\frac{1}{2m+1}\right)
		\right]\\
		&\quad=
		\frac{2}{(2m+1)^2}
		\left(
		3\frac{2\cos\frac{(2m+1)\pi}{3}}{2m+1}
		-H_{2m}
		-\frac{3}{2m+1}
		\right)\\
		&\quad\equiv
		-\frac{3}{4t^3}
		\left(
		2\cos\frac{(p-2t)\pi}{3}-1
		\right)
		-\frac{H_{2t}}{2t^2}
		\pmod p.
	\end{aligned}
\end{equation}
Since \(p-2t\) is odd and \(p-2t\equiv0\pmod3\) if and only if 
\(t\equiv-p\pmod3\), we have
\[
2\cos\frac{(p-2t)\pi}{3}
=
\begin{cases}
	-2, & t\equiv-p\pmod3,\\
	1,  & t\not\equiv-p\pmod3.
\end{cases}
\]
Combining this with \eqref{eq:3.16} and \eqref{eq:key}, we obtain
\begin{equation}\label{eq:3.20}
	\begin{aligned}
		\Res_{z=t^2}\frac{F(z)}{z^2}
		-
		\frac{5(-1)^t}{2t^5}
		&\equiv
		-\frac{3(-1)^t}{2t^5}
		\left(
		2\cos\frac{(p-2t)\pi}{3}-1
		\right)\\
		&\equiv
		\begin{cases}
			0 \pmod p,
			& t\not\equiv-p\pmod3;\\[2mm]
			\dfrac{9(-1)^t}{2t^5} \pmod p,
			& t\equiv-p\pmod3.
		\end{cases}
		\end{aligned}
\end{equation}

\subsection{Completion of the proof}
Combining \eqref{eq:3.5} and \eqref{eq:3.20}, we obtain
\begin{equation*}
S =-\sum_{t=1}^{(p-1)/2}\Res_{z=t^2}\frac{F(z)}{z^2} 
\equiv - \frac{5}{2}\sum_{t = 1}^{(p-1)/2}{\frac{(-1)^{t}}{t^{5}} -}\frac{9}{2}\sum_{\substack{1 \leq t \leq (p-1)/2 \\ t \equiv - p\pmod{3}}}{\frac{(-1)^{t}}{t^{5}}\ }\pmod{p}.
\end{equation*}
Applying Lemma~\ref{lem:harmonic} and \eqref{eq:2.4}, we obtain
\begin{equation}\label{eq:3.21}
S \equiv  - \frac{4}{27}H_{(p-1)/2}^{(5)} \equiv \frac{8}{9}B_{p - 5}\pmod{p}.
\end{equation}
For \(1\le j\le(p-1)/2\), Lemma~2.1 in~\cite{sun2011} gives
\begin{equation*}
j\binom{2j}{j}\binom{2p - 2j}{p - j} \equiv - 2p\pmod{p^2}.
\end{equation*}
Hence
\begin{equation*}
\frac{- 2p}{j\binom{2p - 2j}{p - j}} \equiv \binom{2j}{j}\pmod{p}.
\end{equation*}
Therefore, by the definition of $S$,
\begin{equation*}
- 2\sum_{j = 1}^{(p-1)/2}\frac{p}{j^{4}\binom{2p - 2j}{p - j}}\left( \frac{1}{j^{2}} + H_{j}^{(2)} \right) \equiv \sum_{j = 1}^{(p-1)/2}\frac{\binom{2j}{j}}{j^{3}}\left( \frac{1}{j^{2}} + H_{j}^{(2)} \right) = S \equiv \frac{8}{9}B_{p - 5}\pmod{p}.
\end{equation*}
This proves Theorem~\ref{thm:half}.

\section{Proof of Theorem~\ref{thm:main}}
Let $p>5$ be a prime. Since $\binom{2k}{k}$ is a $p$-adic unit for
$1\le k\le (p-1)/2$, the corresponding summands vanish modulo $p$.
Hence
    \begin{equation*}
		\sum_{k = 1}^{p - 1}{\frac{p}{k^{4}\binom{2k}{k}}\left( \frac{1}{k^{2}} - H_{k - 1}^{(2)} \right)} \equiv \sum_{k = (p + 1)/2}^{p - 1}{\frac{p}{k^{4}\binom{2k}{k}}\left( \frac{1}{k^{2}} - H_{k - 1}^{(2)} \right)}\pmod{p}.
    \end{equation*}
 For $1\le j\le (p-1)/2$, it follows from \eqref{eq:H25} that
\begin{equation*}
H_{p-j-1}^{(2)}
=H_{p-1}^{(2)}-\sum_{k=1}^{j}\frac{1}{(p-k)^2}
\equiv -H_j^{(2)}\pmod p.
\end{equation*}
Reindexing the remaining terms with $k\ge(p+1)/2$ 
by $k=p-j$
 and using 
Theorem~\ref{thm:half}, we obtain
\begin{equation}\label{eq:4.1}
	\begin{aligned}
	\sum_{k = 1}^{p - 1}{\frac{p}{k^{4}\binom{2k}{k}}\left( \frac{1}{k^{2}} - H_{k - 1}^{(2)} \right)} &\equiv  \sum_{j = 1}^{(p - 1)/2}\frac{p}{(p - j)^{4}\binom{2p - 2j}{p - j}}\left( \frac{1}{(p - j)^{2}} - H_{p - j - 1}^{(2)} \right) \\
	&\equiv \sum_{j = 1}^{(p-1)/2}\frac{p}{j^{4}\binom{2p - 2j}{p - j}}\left( \frac{1}{j^{2}} + H_{j}^{(2)} \right) 
	\equiv - \frac{4}{9}B_{p - 5}\pmod{p}.
	\end{aligned}
\end{equation}

For \(1\le k\le p-1\), we have
\begin{equation*}
	\begin{aligned}
		(-1)^{k - 1}\frac{1}{k}\binom{p + k - 1}{2k - 1} 
		&= (-1)^{k - 1}\frac{2p}{(2k)!}\prod_{j = 1}^{k - 1}\left( p^{2} - j^{2} \right) 
		&\equiv \frac{2p}{k^{2}\binom{2k}{k}}\left( 1 - p^{2}H_{k - 1}^{(2)} \right)\pmod{p^3}.
	\end{aligned}
\end{equation*}
Since $p>5$ is prime, we have $p\equiv\pm1\pmod6$. 
Applying Lemma~\ref{lem:En} with $n=p$, we obtain
\begin{equation*}
	\frac{E_{p}}{p} = \sum_{k = 1}^{p - 1}(-1)^{k - 1}\frac{\binom{p + k - 1}{2k - 1}}{k\left( k^{2} - p^{2} \right)} \equiv \sum_{k = 1}^{p - 1}\frac{2p}{k^{4}\binom{2k}{k}}\left( 1 + \frac{p^{2}}{k^{2}} - p^{2}H_{k - 1}^{(2)} \right)\pmod{p^3}.
\end{equation*}
By \eqref{eq:4.1}, we have
\begin{equation}\label{eq:4.2}
	\frac{E_{p}}{p} \equiv \sum_{k = 1}^{p - 1}\frac{2p}{k^{4}\binom{2k}{k}}- \frac{8}{9}p^{2}B_{p - 5}\pmod{p^3}.
\end{equation}
On the other hand,
\begin{equation*}
E_{p} = \sum_{k = 1}^{p - 1}\frac{1}{p^{2} - k^{2}} \equiv - \sum_{k = 1}^{p - 1}{\frac{1}{k^{2}}\left( 1 + \frac{p^{2}}{k^{2}} \right)} = - H_{p - 1}^{(2)} - p^{2}H_{p - 1}^{(4)}\pmod{p^4}.
\end{equation*}
By \eqref{eq:2.3} and \eqref{eq:2.5}, we have
\begin{equation}\label{eq:4.3}
\frac{E_{p}}{p} \equiv - \frac{H_{p - 1}^{(2)}}{p} - pH_{p - 1}^{(4)}\  \equiv \frac{2H_{p-1}}{p^{2}} + H_{p - 1}^{(3)}\  \equiv \frac{2H_{p-1}}{p^{2}} - \frac{6}{5}p^2B_{p-5}\pmod{p^3}.
\end{equation}
Comparing \eqref{eq:4.2} and \eqref{eq:4.3} gives
\begin{equation*}
\sum_{k = 1}^{p - 1}\frac{p}{k^{4}\binom{2k}{k}} - \frac{H_{p - 1}}{p^{2}} \equiv- \frac{7}{45}p^2B_{p-5}\pmod{p^{3}}.
\end{equation*}
Dividing by \(p\) yields \eqref{eq:sun},
and the proof of Theorem~\ref{thm:main} is complete.

\end{document}